\documentclass{amsart}
\usepackage{amssymb}
\usepackage{cleveref}
\usepackage[all]{xy} 
\usepackage{color}
\usepackage{tikz}
\newtheorem{Theo}{Theorem}
\newtheorem{cor}{Corollary} 
\newtheorem{prop}[cor]{Proposition} 
\newtheorem{lem}[cor]{Lemma}
\theoremstyle{definition} 
\newtheorem{defn}[cor]{Definition} 
\newtheorem{rem}[cor]{Remark} 
\newtheorem{ex}[cor]{Example} 

\newcommand{\Z}{\mathbb{Z}}
\newcommand{\MMM}{\mathcal{M}}

\newcommand{\SL}{\mbox{SL}}
\newcommand{\psl}{\mbox{PSL}}

\newcommand{\RR}{\mathbb{R}}
\newcommand{\NN}{\mathbb{N}}
\newcommand{\trans}{\mbox{Trans}}
\newcommand{\emtrans}{\mbox{\em Trans}}
\newcommand{\gal}{\mbox{Gal}}
\newcommand{\slzwei}{\mbox{SL}_2}

\newcommand{\id}{\mbox{id}}
\newcommand{\aut}{\mbox{Aut}}

\newcommand{\OriSquare}[5]
{
\begin{xy}
(0,0)="Pos";
"Pos"+(1,0) **@{-}; ?*h!U!/^2pt/{\text{\scriptsize #5}}, 
"Pos"+(1,1) **@{-}; ?*h!L!/^1pt/{\text{\scriptsize #2}},
"Pos"+(0,1) **@{-}; ?*h!D!/^1pt/{\text{\scriptsize #3}},
"Pos" **@{-}; ?*h!R!/^1pt/{\text{\scriptsize #4}},
"Pos"+(0.5,0.5) *h{#1};
\end{xy}
}
\begin{document}

\title{Finite translation surfaces with maximal number of translations}

\author{J.-C. Schlage-Puchta}
\address{Jan-Christoph Schlage-Puchta: Institute for Mathematics, University of Rostock, 18051 Rostock, Germany}
\email{jan-christoph.schlage-puchta@uni-rostock.de}
\urladdr{}
\author{G. Weitze-Schmith\"usen}
\address{Gabriela Weitze-Schmith\"usen: Institute for Algebra and Geometry, Karlsruhe Institute of Technology (KIT), 76128 Karlsruhe, Germany}
\email{weitze-schmithuesen@kit.edu}

\begin{abstract}
The natural automorphism group of a translation surface
is its group of translations. For finite translation surfaces of genus $g \geq 2$ 
the order of this group is naturally bounded in terms of $g$ due to a Riemann-Hurwitz formula argument.
In analogy with classical Hurwitz surfaces, we call surfaces which achieve the maximal bound 
{\em Hurwitz translation surfaces}. We study for which  $g$ there exist Hurwitz translation surfaces
of genus $g$.
\end{abstract}
\maketitle

\section*{Introduction}

A {\em finite translation surface} is a closed Riemann surface together
with a translation atlas $\mu$, i.e. an atlas all of whose transition maps are translations,
which is defined on $X$ up to finitely many cone points, see \Cref{basics} for a more detailed description
and references. They very naturally show up when studying Teichm\"uller spaces,
since every finite translation surface comes from a pair $(X,\omega)$, where $X$ is a Riemann surface
and $\omega$ is a non zero holomorphic differential. The quadratic differential 
$\omega^2$ defines a point in the cotangent
space of Teichm\"uller space which determines a {\em Teichm\"uller disk}, see e.g. \cite{EG}.  
In his seminal work in \cite{Veech3} Veech studied translation surfaces in the context of billiards.
He gave  a beautiful relation between the geodesic flow on the
translation surface $(X,\omega)$ and a subgroup $\SL(X,\omega)$ of $\slzwei(\RR)$ 
today called {\em Veech group} that can be associated to $(X,\omega)$. 
A particular nice class of translation surfaces are  {\em square-tiled surfaces} also called {\em origamis}
which are obtained by gluing a finite number of squares along their edges via translations such
that the resulting space is a connected surface. 
In general, translation surfaces are distinguished by the types of their singularities.
The cone angle of a singularity is always a multiple of $2\pi$.
The space of all translation surfaces which have $n$ singularities of 
cone angles $(k_1+1)\cdot 2\pi$, \ldots, $(k_n+1)\cdot 2\pi$ 
is called  $H(k_1, \ldots, k_n)$.\\

The natural automorphism group of a translation surface is the group 
of its translations. 
In this article we study the following question. Fix a natural number $g$.
What is the maximal number of translations a translation surface in genus $g$ can have?
This question can be seen as analogon for translation surfaces to a problem that 
Hurwitz studied in the 1890's, namely how many automorphisms 
a Riemann surface of genus $g \geq 2$ can have. He obtained his famous upper bound of $84(g-1)$ 
\cite[Abschnitt II.7,p.424]{Hurwitz}.
This result from 1893 is today known as {\em Hurwitz automorphism theorem}. Since then there has been a vivid study
to describe the Riemann surfaces which achieve this upper bound, see e.g. \cite{Conder1990}
for an overview. They are called
{\em Hurwitz surfaces} and do not occur in all genera, see \cite{Accola1968} and \cite{MacLachlan1969}.
More precisely it is shown in \cite{Larsen2001} that the genera in which Hurwitz surfaces
exist are as rare as perfect cubes; whereas one can deduce from \cite{Schlage-Puchta_Wolfart2006}
that for the good genera there have to be a lot of Hurwitz surfaces.\\

Similarly as for Hurwitz surfaces there is a bound
on the order of  the translation groups of translation surfaces of genus $g$, 
namely $c(g) = 4g-4$. It is achieved only for special $g$'s.
In analogy to Hurwitz's theory we call 
translation surfaces with the maximal possible number $4g-4$ of translations 
{\em Hurwitz translation surfaces}.
It turns out that Hurwitz translation surfaces belong to the special class of translation surfaces 
formed by the origamis. They are even {\em normal origamis} (see \Cref{basics} for the definition). 
Furthermore they always belong to the principal stratum $H(1,\ldots,1)$.
More precisely, we obtain the characterisation of
Hurwitz translation surfaces stated in Theorem~\ref{Theo-criterion}

\begin{Theo}[proven in \Cref{section-tHs}] \label{Theo-criterion}
Let $g \geq 2$.
\begin{enumerate}
\item[i)]
  A finite translation surface $(X,\mu)$ of genus $g$ has at most $4g - 4$ translations.
  It has precisely $4g-4$ translations if and only if 
  $(X,\mu)$ is a normal origami in the stratum $H(1,\ldots, 1)$. 
\item[ii)]
  A finite group $G$ is the automorphism group of a Hurwitz translation surface
  if and only if it can be generated by two elements $a$ and $b$ such 
  that their commutator $[a,b]$ has order 2.
\end{enumerate}
\end{Theo}

We then study in which genus there exist Hurwitz translation surfaces
and obtain the answer to this question in Theorem~\ref{Theo-tHnumbers}.

\begin{Theo}[proven in \Cref{section-tHn}]\label{Theo-tHnumbers}
There exists a Hurwitz translation surface
of genus $g$ if and only if $g$ is odd or $g-1$ is divisible by 3.
\end{Theo}

{\bf Acknowledgement:}
The second author would like to thank the F{\'e}d{\'e}ration de recherche des Unit{\'e}s de Math{\'e}matiques de Marseille
for its hospitality and the stimulating working environment and 
the Institute for Computational and Experimental Research in Mathematics (ICERM) for the wonderful and intensive
semester program {\em Low-dimensional Topology, Geometry, and Dynamics 2013} during which part of this project was accomplished.

\section{Basics}\label{basics}

In this section we give the basic definitions on translation surfaces and origamis which we will use. 
More comprehensive introductions can be found e.g. in \cite{HeSc}, \cite{HS1}, \cite{Sc1} and \cite{Zorich}. 
A {\em translations surface} $(X^*,\mu)$ is a surface together with
an atlas $\mu$ all of whose transition maps are translations. In this article
we restrict ourselves to {\em finite translation surfaces} $(X,\mu)$, i.e. $X$ is a closed
surface, $\mu$ is a translation atlas on $X^* = X\backslash\{P_1,\ldots,P_n\}$
where $P_1$, \ldots, $P_n$ are finitely many points of $X$, and all 
of the $P_i$'s are finite cone angle singularities of some cone angle $k_i\cdot 2\pi$.  
Such translation surfaces can always be obtained by the following
handy construction: Take finitely many polygons in the Euclidean plane $\RR^2$
such that their edges come in pairs which are
parallel and have the same length. Identify for each pair these edges via a translation such
that you obtain a connected oriented surface $X$. By construction $X$ carries the structure
of a finite translation surface possibly having singularities in the points which 
come from the vertices of the polygons.

\begin{center}
\newdimen\R
\R=1cm
\begin{figure}[h]
\begin{tikzpicture}
    \draw (0:\R) \foreach \x in {72,144,...,359} {
            -- (\x:\R)
        } -- cycle (90:\R) node[above] {} ;
    \draw[xshift=1.5\R,yshift=1.03\R] (36:\R) \foreach \x in {36,108,180,252,324} {
            --  (\x:\R) 
        } --  cycle (90:\R) node[above] {} ;
    \node at (108:\R){\small a};  \node at (180:\R){\small b};  \node at (252:\R){\small c};  \node at (324:\R){\small d};
    \draw[xshift=1.5\R,yshift=1.03\R] node at (0:\R){\small b}  node at (72:\R){\small c} node at (144:\R){\small d} node at (-72:\R){\small a}; 
\end{tikzpicture}
\caption{Translation surface obtained by gluing parallel edges of two regular pentagons}\label{dbpentagon}
\end{figure}
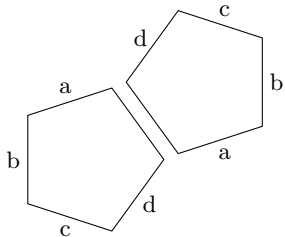
\end{center}

\Cref{dbpentagon} shows the translation surfaces built from two regular pentagons. It belongs to the famous double regular $n$-gon series 
which was already studied by Veech in \cite{Veech3}. In the case of the double pentagon all vertices of the two polygons glue to a
single point of the surface. Thus the Euler characteristic is $\chi = 1-5+2 = -2$ and the genus is $2$.  The cone angle around the
one singular point obtained from the 10 vertices is $6\pi$, which is the sum of all interior angles of the two pentagons.\\

The set of all finite translation surfaces of genus $g$ forms a space called $\Omega\MMM_g$, see e.g. \cite{KontsevichZorich}. 
This space is stratified 
by the data of the 
angles of the singularities. More precisely the stratum $\Omega\MMM_g(k_1,\ldots,k_n)$ consists
of the finite translation surfaces of genus $g$ with $n$ singularities of angle $(k_1+1)\cdot 2\pi$, \ldots, $(k_n+1)\cdot 2\pi$.  
An easy Euler characteristic calculation shows that $k_1 + \ldots + k_n = 2g-2$. It follows in particular
that the genus of a translation surface is at least 1.\\

A very special class of translation surfaces are origamis (also called
{\em square-tiled surfaces}). There are several equivalent definitions for origamis, see e.g. \cite[Section 1]{Sc2}.
We will use the following ones.

\begin{defn}\label{def-origami}
  An {\em origami} $O$ is equivalently given by one of the three following objects:
  \begin{enumerate}
  \item\label{fstdef}
    A translation surface $(X,\mu)$  obtained by taking finitely many 
    copies of the Euclidean unit squares and gluing their edges via translations, up to equivalence.
    Two such surfaces $(X_1,\mu_1)$ and $(X_2,\mu_2)$  are equivalent, if there exists a 
    homeomorphism $f:X_1 \to X_2$
    which is a translation with respect to the charts in $\mu_1$ and $\mu_2$.
  \item\label{snddef}
    A finite degree cover $p:X \to E$ of the torus $E$ which 
    is ramified at most over one point, up to equivalence. Two such covers $p_1:X_1 \to E_1$
    and $p_2: X_2 \to E_2$ are equivalent, if there exists a homeomorphism $h: X_1 \to X_2$
    with $p_2 \circ h = p_1$.
  \item\label{thrddef}
    A pair of permutations $(\sigma_a,\sigma_b)$
    in the symmetric group $S_d$ with some $d \in \NN$ up to equivalence by
    simultaneous conjugation.
  \end{enumerate}
\end{defn} 

The equivalence of these definitions is e.g. described in \cite[Section 1]{Sc2}.
We here just explain them for the example shown in \Cref{fig-wms}.
The number $d$ in (\ref{thrddef}) of \Cref{def-origami} is the number of the squares in 
(\ref{fstdef})
and the degree of the covering in (\ref{snddef}).\\

\begin{figure}[h]
\begin{center}
\hspace*{1mm}
    \begin{xy}
      <1.3cm,0cm>:
      (0,1)*{\OriSquare{1}{}{}{4}{6}};
      (1,1)*{\OriSquare{2}{}{7}{}{}};
      (2,1)*{\OriSquare{3}{}{}{}{8}};
      (3,1)*{\OriSquare{4}{1}{5}{}{}};
      (1,0)*{\OriSquare{5}{6}{}{8}{4}};
      (2,2)*{\OriSquare{6}{7}{1}{5}{}};
      (3,0)*{\OriSquare{7}{8}{}{6}{2}};
      (0,2)*{\OriSquare{8}{5}{3}{7}{}};
      (-0.5,0.52)*{{\color{red}\bullet}};
      (3.5,0.52)*{{\color{red}\bullet}};
      (-0.5,2.52)*{{\color{red}\bullet}};
      (1.5,.52)*{{\color{red}\bullet}};
      (1.5,2.52)*{{\color{red}\bullet}};
      (0.5,0.52)*{{\color{green}\bullet}};
      (2.5,2.52)*{{\color{green}\bullet}};
      (2.5,0.52)*{{\color{green}\bullet}};
      (.5,2.52)*{{\color{green}\bullet}};
      (-0.5,1.52)*{{\color{blue}\bullet}};
      (3.5,1.52)*{{\color{blue}\bullet}};
      (1.5,-.48)*{{\color{blue}\bullet}};
      (1.5,1.52)*{{\color{blue}\bullet}};
      (3.5,-.48)*{{\color{blue}\bullet}};
      (0.5,1.52)*{{\color{brown}\bullet}};
      (2.5,-.48)*{{\color{brown}\bullet}};
      (2.5,1.52)*{{\color{brown}\bullet}};
      (.5,-.48)*{{\color{brown}\bullet}};
    \end{xy}
 
    \caption{A normal origami of genus 3 called the {\em Eierlegende Wollmilchsau}}\label{fig-wms}
 \end{center} 
\end{figure}
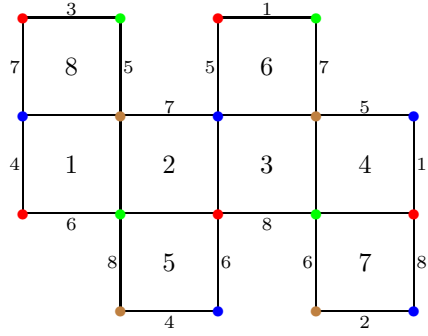

\Cref{fig-wms} shows an origami using the first version in \Cref{def-origami}.
A right edge labelled by the number $a$ is glued to the left edge of the square labelled by $a$
and similarly for the upper and the lower edges.
Using the Euler characteristic formula one easily calculates that it is of genus 3. 
You directly see the corresponding covering $p$ from this picture: Present the torus $E$
as the surface obtained by gluing opposite edge of the Euclidean unit square. Then the cover $p$
maps each square in \Cref{fig-wms} to the one square forming the torus. This gives
a well-defined cover of the corresponding surface to $E$ which is ramified
only over the one point obtained from the four vertices of the unit square. Also the
pair of permutations in the third description of \Cref{def-origami}
can be easily read off from the picture. $\sigma_a$ is the permutation which maps
the number of a square to the number of its right neighbour; the permutation
$\sigma_b$ maps it to the number of its upper neighbour. Thus we obtain
 \begin{equation}
   \sigma_a = (1,2,3,4)(5,6,7,8), \quad \sigma_b = (1,8,3,6)(2,7,4,5). \label{ex-wms}
 \end{equation}
This origami is called {\em Eierlegende Wollmilchsau} (see \cite{HeSc2008})
since it has several nice properties and has served as counter examples in several occasions, see e.g. 
\cite{Forni2006}, \cite{MatheusYoccoz2010} and \cite{Moeller2011}.\\

A special class of origamis are {\em normal origamis}. They are by definition
origamis $p:X \to E$ for which 
the covering $p$ is normal, i.e. it is the quotient by a group $G$. In this case the group $G$ is
finite, its order is the degree $d$ of the covering and it is the group $\gal(X/E)$
of deck transformations of the covering $p$, i.e. homeomorphisms $h$ with $p\circ h = p$. 
In particular the group $G$ acts 
transitively on the fibre of a point of the torus $E$. We may use this to label the squares by the elements
in $G$: Label one square $Sq$ by $\id$ and label the square $g\cdot Sq$  by  $g$. In particular, if the right 
neighbour of Square $\id$
is labelled by $a$ and the left neighbour by $b$, then the dual graph of the origami
is the Cayley graph of $G$ with respect to the generators $a$ and $b$. Here the dual graph of the origami
is the following graph. The vertices of the graph correspond to the squares of the origami
and the edges of the graph are labelled by $a$ and $b$: two vertices in the graph are
connected by an $a$-edge, if the corresponding squares intersect in a vertical edge,
and similarly for the $b$-edges of the graph and the horizontal edges of the origami.

\begin{ex}\label{ex-wms-normal}
Observe that the origami shown in \Cref{fig-wms} is normal. More precisely the group $G$
of deck transformations is the quaternion group 
\[Q = \{\pm i, \pm j, \pm k, \pm 1\} = \; <i,j,-1| k = ij = -ji, i^2 = j^2 = k^2 = -1, (-1)^2 = 1>\]
\hspace*{1.3cm}

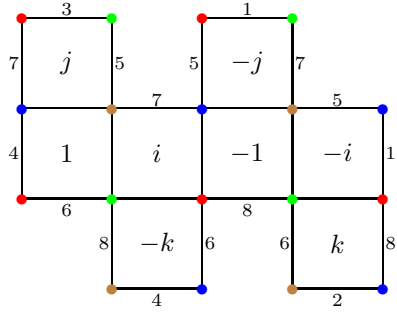
\begin{figure}[h]
\hspace*{1mm}
\begin{xy}
<1.2cm,0cm>:
(0,1)*{\OriSquare{1}{}{}{4}{6}};
(1,1)*{\OriSquare{i}{}{7}{}{}};
(2,1)*{\OriSquare{-1}{}{}{}{8}};
(3,1)*{\OriSquare{-i}{1}{5}{}{}};
(1,0)*{\OriSquare{-k}{6}{}{8}{4}};
(2,2)*{\OriSquare{-j}{7}{1}{5}{}};
(3,0)*{\OriSquare{k}{8}{}{6}{2}};
(0,2)*{\OriSquare{j}{5}{3}{7}{}};
(-0.5,0.52)*{{\color{red}\bullet}};
(3.5,0.52)*{{\color{red}\bullet}};
(-0.5,2.52)*{{\color{red}\bullet}};
(1.5,.52)*{{\color{red}\bullet}};
(1.5,2.52)*{{\color{red}\bullet}};
(0.5,0.52)*{{\color{green}\bullet}};
(2.5,2.52)*{{\color{green}\bullet}};
(2.5,0.52)*{{\color{green}\bullet}};
(.5,2.52)*{{\color{green}\bullet}};
(-0.5,1.52)*{{\color{blue}\bullet}};
(3.5,1.52)*{{\color{blue}\bullet}};
(1.5,-.48)*{{\color{blue}\bullet}};
(1.5,1.52)*{{\color{blue}\bullet}};
(3.5,-.48)*{{\color{blue}\bullet}};
(0.5,1.52)*{{\color{brown}\bullet}};
(2.5,-.48)*{{\color{brown}\bullet}};
(2.5,1.52)*{{\color{brown}\bullet}};
(.5,-.48)*{{\color{brown}\bullet}};
\end{xy}
\caption{Squares of the translation surface from \Cref{fig-wms}
         can be labelled by the elements in the quaternion group $Q$.}
\end{figure}
\end{ex}

We finish this section by giving a formal definition of the
main actor of this article.

\begin{defn}
Let $(X,\mu)$ be a finite translation surface.
The {\em group of translations} $\emtrans(X,\mu)$ is the group of the homeomorphisms
which are translations with respect to $\mu$,  i.e. with respect to the charts of $\mu$ they
have the form ${x\choose y} \mapsto {x\choose y} + {c_1\choose c_2}$ for some constant vector 
${c_1\choose c_2}$ depending on the charts.
\end{defn}

\section{Hurwitz translation surfaces}\label{section-tHs}

In this section we show that a translation surface of genus $g$ can have at most $4g-4$
automorphisms. We call a translation surface {\em Hurwitz translation surface}, if it achieves this
upper bound. We prove the criterion in \Cref{Theo-criterion} 
for translation surfaces to be Hurwitz translation surface and give first 
examples for Hurwitz translation surfaces.

\begin{lem}\label{lemma-Htrans}
Let $X$ be a precompact translation surface of genus $g \geq 2$. $X$ has
at most $c(g) = 4g -4$ translations. If $X$ has
$4g-4$ translations, then it lies in the stratum $H(1,\ldots, 1)$  
and is an origami.
\end{lem}

\begin{proof}
Let $X$ be a precompact translation surface of genus 
$g \geq 2$, $G = \trans(X)$ its group
of translations and $d$ the order of $G$. 
Consider the ramified covering $p:X \to X/G$. 
$X/G$ is again a precompact translation surface, hence
in particular its genus $g_{X/G}$ is greater or equal to $1$.
Let $P_1$, \ldots, $P_k$ on $X/G$ 
be the ramification points of $p$. Since $p$ is normal,
all preimages of $P_i$ have the same ramification
index $e_i$. Let $s_i$ be the number of preimages
of $P_i$, thus we have $d = s_ie_i$ and $s_i \leq \frac{d}{2}$. 
By Riemann-Hurwitz we have
for the genus $g$ and $g_{X/G}$ of $X$ and $X/G$, respectively:
\begin{equation}\label{RHformula}
\begin{array}{lcl}
2g-2 &=& d(2g_{X/G}-2) + \sum_{i=1}^k s_i(e_i-1)
     =  d(2g_{X/G}-2) + kd - \sum_{i = 1}^ks_i\\
     &\geq& d(2g_{X/G}-2)  +kd - k\cdot \frac{d}{2} = d(2g_{X/G}-2+\frac{k}{2})
\end{array}
\end{equation}
Observe that since $g_{X/G} \geq 1$, we have that
 $2g_{X/G}-2+\frac{k}{2} = 0$ if and only if
$g_{X/G} = 1$ and $k = 0$. In this case we obtain from 
(\ref{RHformula}) that $g = 1$ which contradicts our assumption.
Hence we have
\[ d(2g_{X/G}-2+\frac{k}{2}) \leq 2g-2\]
with $ 2g_{X/G}-2+\frac{k}{2} \geq \frac{1}{2}$ and thus
$d \leq 4g-4$. 
Equality holds if and only if $ 2g_{X/G}-2+\frac{k}{2}= \frac12$
and $s_i = \frac{d}{2}$ for all $i$, i.e. if $g_{X/G} = 1$, $k = 1$ and
$s_1 = \frac{d}{2}$. Thus $p$ is a covering of the torus
ramified over one point and the ramification index 
of the preimages are all $e_i = 2$. In particular $X$ is an origami
in the stratum $H(1,\ldots, 1)$.
\end{proof}

\begin{defn}
We say that a translation surface $X$ {\em is a Hurwitz translation surface (Hts)},
if $X$ has $4g-4$ translations. 
\end{defn}

\begin{cor}
It directly follows from the proof of \Cref{lemma-Htrans} that a translation surface
is a Hts if and only if it is a normal origami in the stratum $H(1,\ldots,1)$.
\end{cor}

\begin{ex}\label{Hts-examples}
  The following origamis are Hurwitz translation surfaces:
  \begin{enumerate}
  \item
    The {\em Eierlegende Wollmilchsau} from \Cref{ex-wms-normal} is a Hurwitz translation surface.
   \item
    The {\em Escalator with $8$ squares} (see Figure~\ref{fig-elevator}) defined by the permutations:
    \[\sigma_a = (1,2)(3,4)(5,6)(7,8), \quad \sigma_b = (2,3)(4,5)(6,7)(8,1)\]
    The automorphism group is the dihedral group 
    \[D_4 = <\tau_1,\tau_2; \tau_1^2,\tau_2^2, (\tau_1\tau_2)^4>\] 
    of 8 elements. Going to the right corresponds to multiplication by $\tau_1$,
    going up corresponds to multiplication by $\tau_2$. The origami has four singularities 
    of total angle $4\pi$ and is thus of genus 3.
    \begin{center}    
      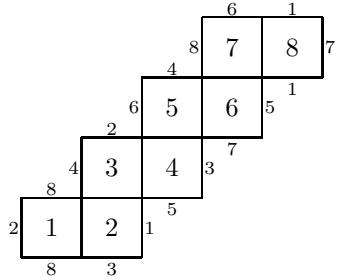
\begin{figure}[h]
        \hspace*{2mm}
        \begin{xy}
          <.8cm,0cm>:
          (0,0)*{\OriSquare{1}{}{8}{2}{8}};
          (1,0)*{\OriSquare{2}{1}{}{}{3}};
          (1,1)*{\OriSquare{3}{}{2}{4}{}};
          (2,1)*{\OriSquare{4}{3}{}{}{5}};
          (2,2)*{\OriSquare{5}{}{4}{6}{}};
          (3,2)*{\OriSquare{6}{5}{}{}{7}};
          (3,3)*{\OriSquare{7}{}{6}{8}{}};
          (4,3)*{\OriSquare{8}{7}{1}{}{1}};
        \end{xy}
        \caption{The Escalator: a normal origami of genus 3 with 8 translations}\label{fig-elevator}
    \end{figure}
    \end{center}
  \item\label{A4}
    The origami given by the following two permutations, see Figure~\ref{A4-origami}:
    \begin{align*}
      \sigma_a &= (1,5,7)(2,4,8)(11,12,10)(3,6,9),\\
      \sigma_b &= (1,4)(2,6)(3,5)(7,11)(8,10)(9,12).
    \end{align*}
    The origami is of genus 4. 
    Its automorphism group is the alternating group $A_4$. The squares correspond to the elements in $A_4$
    as follows: 
    \begin{align*}
     1 &\leftrightarrow \id, 2  \leftrightarrow (1,4,3), 3  \leftrightarrow (1,3,4), 4 \leftrightarrow (1,2)(3,4),\\
     5 &\leftrightarrow (1,2,3), 6  \leftrightarrow (1,2,4), 7  \leftrightarrow (1,3,2), 8  \leftrightarrow (2,4,3)\\
     9 &\leftrightarrow (1,4)(2,3), 10 \leftrightarrow (1,4,2), 11 \leftrightarrow (2,3,4), 12 \leftrightarrow (1,3)(2,4)
    \end{align*}
    \begin{center}    
      \begin{figure}[h]
        \hspace*{2mm}
        \begin{xy}
          <.8cm,0cm>:
          (0,1)*{\OriSquare{1}{}{4}{7}{4}};
          (5,1)*{\OriSquare{2}{}{6}{}{}};
          (1,0)*{\OriSquare{3}{6}{}{9}{5}};
          (6,1)*{\OriSquare{4}{8}{1}{}{1}};
          (1,1)*{\OriSquare{5}{}{3}{}{}};
          (5,0)*{\OriSquare{6}{9}{}{3}{2}};
          (2,1)*{\OriSquare{7}{1}{}{}{11}};
          (4,1)*{\OriSquare{8}{}{}{4}{10}};
          (3,3)*{\OriSquare{9}{3}{12}{6}{}};
          (4,2)*{\OriSquare{10}{11}{8}{}{}};
          (2,2)*{\OriSquare{11}{}{7}{10}{}};
          (3,2)*{\OriSquare{12}{}{}{}{9}};
        \end{xy}
        \caption{A normal origami in genus 4 whose deck group is the alternating group $A_4$} \label{A4-origami}
      \end{figure}
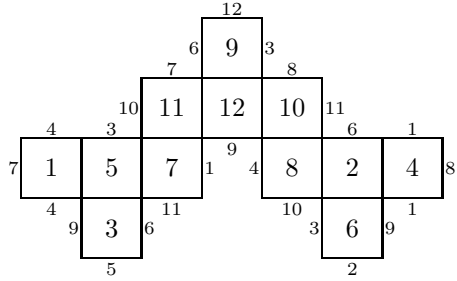
    \end{center}
  \end{enumerate}
\end{ex}

\begin{lem}\label{thegroup}
Let $G$ be a finite group of order $d$ which is generated by two elements
$a$ and $b$. Suppose that the commutator $[a,b]$ has order 2.
$G$ acts on itself by multiplication from the right.
Identify the elements of $G$ with the numbers of $\{1,\ldots, d\}$. 
Then each element $g$ of $G$ defines a permutation $\sigma_g$ in $S_d$.
Let $O$ be the origami defined by the pair of permutations $(\sigma_a,\sigma_b)$.
Then $O$ is a Hurwitz translation surface and 
we obtain any Hurwitz translation surface in this way. 
\end{lem}

\begin{proof}
The squares of the origami $O$ correspond to the elements of $G$.
The right neighbour of the square labelled by $g$ is $g\cdot a$ 
and the upper neighbour is $g\cdot b$.
Let $p:X \to E$ be the corresponding covering to the torus.
Consider a small simple closed loop on the torus around the ramification point which goes first to the right, then up,
then to the left and then down. 
If we lift this loop to $X$ via $p$ in a point which lies let us say in 
the square labelled by $g \in G$, then the lift ends in a point which lies in the square
labelled by  
$g\cdot a\cdot b\cdot a^{-1}\cdot b^{-1}$. Hence lifting the second power of our simple 
closed loop closes up.
This is true for any $g$, thus the ramification
index of each point above the puncture of the torus is $2$ and hence all
singularities of $O$ are of total angle $4\pi$. Thus $O$ is 
a normal origami in $H(1,\ldots, 1)$. Conversely, if we start with a normal
origami in $H(1,\ldots,1)$, its group of 
translations which is equal to the deck transformation group has the desired 
property. Hence the origami defines a Hurwitz translation surface.
\end{proof}

\Cref{Theo-criterion} now directly follows from \Cref{lemma-Htrans} and \Cref{thegroup}.
Furthermore, we obtain the following non-example of \Cref{thegroup}.

\begin{rem}\label{non-example}
There are no Hurwitz translation surfaces of genus 2.
\end{rem}

\begin{proof}
Suppose there was a normal origami $X$ in $H(1,1)$.
Then $G = \trans(X)$ has order $4\cdot 2 - 4 = 4$
and is hence abelian. But then the commutator
of two elements is trivial which contradicts \Cref{thegroup}. 
\end{proof}

\section{Translation Hurwitz Numbers}\label{section-tHn}

We have seen in \Cref{Hts-examples} Hurwitz translation surfaces 
in genus 3 and genus 4 and in \Cref{non-example} that there are none in genus 2.
This naturally leads to the question in which genus there exist translation Hurwitz surfaces. This 
becomes via \Cref{thegroup} a purely group theoretical question.
This section is devoted to 
the proof of  \Cref{Theo-tHnumbers}, which answers to this question.\\

We denote the cyclic group with $n$ elements by $C_n$, the 
commutator subgroup of a group $G$ by $G'$, the normaliser of a subgroup $U$ of $G$ by $N_G(U)$
and the index of $U$ by $(G:U)$.

\begin{defn}
A finite group $G$ is called {\em tH (translation Hurwitz)}, 
if there exist two elements $g, h\in G$ such that $\langle g, h\rangle=G$, and $[g, h]$ has order 2. 
An integer $n$ is called {\em tH}, if there exists a tH group of order $n$.
\end{defn}

\begin{prop}\label{tH}
An integer is tH if and only if $n$ is divisible by 8 or 12.
\end{prop}

Thus we can read off from the order of the group, that a group
is not tH. As we will see below it 
is crucial  that 
in the non-cases, i.e. if the group is not divisible
by 8 and not divisible by 12, this implies that the group is solvable. Relating
arithmetic properties of the order of a group to its 
structural properties is an old idea
which is carried out  e.g. in \cite{Pazderski1959}, where they deduce interesting properties
of a group 
as nilpotency and supersolvability just from its order.\\
 
One direction of \Cref{tH} follows from the following examples.

\begin{prop}\label{tH-constructions}
\begin{enumerate}
\item \label{power_two} For $a\geq 3$ we have that $2^a$ is tH.
\item \label{second_case} For $b\geq 1$ we have that $4\cdot3^b$ is tH.
\item \label{composite} If $n$ is tH, and $(n,m)=1$, then $nm$ is tH.
\end{enumerate}
\end{prop}

\begin{proof}
We provide an explicit construction for each claim.\\[3mm]
(\ref{power_two}) 
Consider the group $G = C_{2^{a-1}} \rtimes_{\varphi} C_2$, where $\varphi:C_2 \to \aut(C_{2^{a-1}})$ 
is defined by $\varphi(1): 1 \mapsto 2^{a-2}+1$. As pair of generators we choose the
standard generators $x = (1,0)$ and $y = (0,1)$. Then 
$[x,y] = (2^{a-2},0)$ has order 2.\\[3mm]
(\ref{second_case}) 
The commutator of $A_4$ is $C_2\times C_2$, thus any two elements $x,y$ of $A_4$ which do not commute satisfy that $[x,y]$ has order 2. Now 
$A_4$ is generated by $(1,2,3), (1,2)(3,4)$, thus $A_4$ is tH, and our claim holds for $b=1$ (compare \Cref{Hts-examples}.\ref{A4}). For higher values of $b$ consider $G=A_4\times C_{3^{b-1}}$, and put $x=((1,2,3), 0)$, $y=((1,2)(3,4), 1)$. Then $[x,y]=((1,4)(2,3), 0)$ is of order 2. Moreover, 
$\langle x,y\rangle$ contains $y^{3^{b-1}+1} = (\mathrm{id}, 1)$
as well as $y^{3^{b-1}}=((1,2)(3,4), 0)$, thus $\langle x,y\rangle=G$, and we conclude that $G$ is tH.\\[3mm]
(\ref{composite}) 
Let $G$ be a tH group of order $n$, and let $x,y$ be suitable generators of $G$. Consider $G\times C_m$, and put $x'=(x,0)$, $y'=(y, 1)$. Since $(n,m)=1$ there exist $u, v\in\Z$ such that $un+vm=1$, thus $y'^{un}=(y^{un}, un) = (1_G, 1)$, and we conclude that $\langle x', y'\rangle$ contains $C_m$. Similarly $y'^{vm}=(y^{vm}, vm) = (y, 0)$, thus $\langle x', y'\rangle$ contains $G$. Hence $\langle x', y'\rangle = G\times C_m$, and we conclude that $nm$ is tH.
\end{proof}

It follows from Thompson's classification in \cite{Thompson} of minimal finite simple groups that
every group of order not divisible by 8 or 12 is solvable, as described in the following: The classification in 
\cite[Section 3, Main Theorem]{Thompson} tells us
that we have the following list of all minimal finite simple groups: $\psl(2,2^p)$ with $p$ a prime, $\psl(2,3^p)$ with $p$ an odd prim,
$\psl(2,p)$ with $p$ prime, $p \geq 5$ and $p^2+1 \equiv 0 \mod 5$, a Suzuki group $Sz(2^p)$ with $p$ an odd prime,
or $\psl(3,3)$. Thus the number of elements of each minimal simple group is divisible by 8 or 12. Minimal simple groups
are by definition non-abelian simple groups whose proper subgroups are all solvable.
Thus any non-abelian simple group is either minimal or has a subquotient of smaller order
which is again non-abelian and simple. Iterating this will 
finally lead to a minimal simple group and thus
by induction also the order of each finite non-abelian simple group 
is divisible by 8 or 12.
Finally, any finite non-solvable group $G$ has a subquotient which is non-abelian and simple, thus
the order of $G$ is as well divisible by 8 or 12.\\

In particular we may assume for the proof of the reverse
direction of \Cref{tH} that $G$ is solvable. We shall repeatedly use the
following, confer e.g. \cite[Theorem~VI.1.7]{Huppert}.
\begin{lem}[Hall]
Let $G$ be a finite solvable group, $\pi$ a set of prime divisors of $|G|$. Write $|G|=nm$, where all prime divisors of $n$ are in $\pi$, and all prime divisors of $m$ are not in $\pi$. Then there exists a subgroup $U$ of $G$ with $|U|=n$, and all subgroups of this form are conjugate.
\end{lem}
We call a subgroup of this form a $\pi$-Hall group. If $\pi$ consists of a single prime $p$, we write $p$ in place of $\{p\}$, and for a set $\pi$ we denote by $\pi'$ the complement of $\pi$. In particular, $2'$ denotes the set of all odd primes.
\begin{lem}
\label{Lem:not 2 mod 4}
Suppose that $n$ is even, but not divisible by 4. Then $G$ is not tH
\end{lem}
\begin{proof}
If $|G|\equiv 2\pmod{4}$, then $G$ is solvable, since $|G|$ is not divisible by 4. Let $U$ be a $2'$-Hall group. Then $(G:U)=2$, thus $U$ is normal in $G$, and $G$ projects onto $C_2$. In particular $G'\leq U$, thus $(G:G')$ is even, and $|G'|$ is odd. But then $[x,y]\in G'$ cannot have even order.
\end{proof}

The following statement can be concluded e.g. from \cite[Theorem 9.3.1]{HallTheTheoryOfGroups}. We include the proof for the
convenience of the reader. It mimics the proof of Sylow's theorem. 
 
\begin{lem}
Let $\pi$ be a set of primes, $G$ a solvable group, $U$ a $\pi$-Hall group. Let $n$ be the number of conjugates of $U$. Then $n$ divides the $\pi'$-part of $|G|$, and there exist non-negative integers $a_p$, $p\in\pi$, such that $1+\sum_{p\in \pi} a_p p=n$.
\end{lem}
\begin{proof}
$G$ acts on $\Omega=\{U^g|g\in G\}$ by conjugation. This action is transitive, thus $|\Omega|$ equals the index of a point stabiliser in $G$. The stabiliser of $U$ is $N_G(U)$, hence $|\Omega|=(G:N_G(U))$. Since $U\leq N_G(U)$, we have that $|\Omega|$ divides $(G:U)$, which is the $\pi'$-part of $|G|$.

The action of $U$ on $\Omega$ induces a decomposition of $\Omega$ into $U$-orbits. $U$ itself is stable under $U$-conjugation, hence $\{U\}$ is an orbit of size 1. We claim that there exists no further orbit of size 1. Suppose that $U^g$ is stabilised by $U$. Then $U$ normalises $U^g$, hence the group $\langle U, U^g\rangle$ generated by $U$ and $U^g$ contains $U^g$
as a normal subgroup, and  $\langle U, U^g\rangle/ U^g\cong  U/(U^g\cap U)$. Hence $|\langle U, U^g\rangle| = \frac{|U|^2}{|U\cap U^g|}$ is a $\pi$-number, which is strictly larger than $|U|$, since $U\neq U^g$. But $|U|$ is the largest $\pi$-number dividing $G$, thus $|\langle U, U^g\rangle|$ does not divide $|G|$, which is impossible. Hence we conclude that there exists precisely one orbit of size 1.

The size of every orbit divides $|U|$, hence the size of every orbit is either 1 or divisible by some prime divisor in $\pi$. For each orbit $o$ we pick one such prime divisor $p_o$, and we obtain an equation of the form $|\Omega|= 1 + \sum_{o \mbox{\footnotesize\ Orbit}} p_o \frac{|o|}{p_o}$. Collecting the integers $\frac{|o|}{p_o}$ we obtain the integers $a_p$ from our claim.
\end{proof}
\begin{cor}
Suppose that $|G|$ is divisible by 4, but not by 8 or 12. Then $G$ has a normal subgroup of index 4.
\end{cor}
\begin{proof}
The divisibility condition implies that $|G|$ is solvable. Let $U$ be a $2'$-Hall group. Then $U$ has index 4. Let $n$ be the number of conjugates of $U$. Then $n$ divides 4, and we can write $n$ as
\[
n=1+\underset{p>2}{\sum_{p||G|}} a_p p = 1+\sum_{p\geq 5} a_p p,
\]
since by assumption $|G|$ is not divisible by 3. But every integer representable as a sum as on the right is either 0 or $\geq 5$, hence we obtain $n=1$. But then $U$ is normal, and our claim follows.
\end{proof}

We can now finish the proof of \Cref{tH}.
\begin{proof}(of \Cref{tH})
The reverse direction follows from \Cref{tH-constructions}.
Suppose now that $n$ is tH, and let $G$ be a tH-group of order $n$. Then $n$ is trivially even. The case $n\equiv 2\pmod{4}$ is excluded by \Cref{Lem:not 2 mod 4}, thus $n\equiv 0\pmod{4}$. If $n\equiv 0\pmod{8}$, our claim is true, hence we may assume $n\equiv 4\pmod{8}$. If $n$ is not divisible by 3, then from the corollary we see that $G$ contains a normal subgroup $U$ of index 4. Since every group of order 4 is abelian, we obtain $G'\leq U$. But then $(G:G')$ is divisible by 4, thus $|G'|$ is odd, and we conclude that there cannot be elements $x,y\in G$ such that $[x,y]$ has even order. 
Hence $|G|$ is divisible by 3, and the claim is proven.
\end{proof}

Finally,  \Cref{Theo-tHnumbers} directly follows from \Cref{thegroup} and \Cref{tH}.
\bibliographystyle{amsalpha}
\bibliography{trans}

\end{document}